\documentclass[12pt,leqno]{amsart}
\usepackage{tikz}
\usetikzlibrary{automata,positioning}
\usepackage{amssymb,amsthm,amsmath,latexsym}
\usepackage{graphicx} 
\graphicspath{{figures/}} 
\usepackage{booktabs} 
\usepackage[font=small,labelfont=bf]{caption} 
\usepackage{amsfonts, amsmath, amsthm, amssymb} 
\usepackage{wrapfig} 
\usepackage{multicol}

\newtheorem{theorem}{\sc Theorem}[section]
\newtheorem{thm}[theorem]{\sc Theorem}
\newtheorem{lemma}[theorem]{\sc Lemma}
\newtheorem{proposition}[theorem]{\sc Proposition}
\newtheorem{corollary}[theorem]{\sc Corollary}

\newtheorem{ex}[theorem]{\sc Example}

 \newtheorem*{thmA}{Theorem A}
 \newtheorem*{thmB}{Theorem B}
 \newtheorem*{thmC}{Theorem C}
 \newtheorem*{thmD}{Theorem D}

\title{On wreath product occurring as subgroup of automata group}

\author{Alex C. Dantas}
\address{Departamento de Matem\'atica, Universidade de Bras\'ilia,
Brasilia-DF, 70910-900 Brazil}
\email{(Dantas) alexcdan@gmail.com}

\author{Junio R. Oliveira}
\address{Departamento de Matem\'atica, Universidade de Bras\'ilia,
Brasilia-DF, 70910-900 Brazil}
\email{(Oliveira) junrocoli@gmail.com}

\author{Tulio Gentil}
\address{Departamento de Matem\'atica, Centro de Tecnologia, Cidade Universit\'aria da Universidade Federal do Rio de Janeiro, Rio de Janeiro-RJ, 21941-909 Brazil}
\email{(Gentil) tuliogentil@im.ufrj.br}

\subjclass[2020]{20E08, 20B27, 20K20.}
\keywords{}

\begin{document}
\maketitle

\begin{abstract}
A finitely generated group is said to be an automata group if it admits a faithful self-similar finite-state representation on some regular $m$-tree. We prove that if $G$ is a subgroup of an automata group, then for each finitely generated abelian group $A$, the wreath product $A \wr G$ is a subgroup of an automata group. We obtain, for example, that $C_2 \wr (C_{2} \wr \mathbb{Z})$, $\mathbb{Z} \wr (C_2 \wr \mathbb{Z})$, $C_2 \wr (\mathbb{Z} \wr \mathbb{Z})$, and $\mathbb{Z} \wr (\mathbb{Z} \wr \mathbb{Z})$ are subgroups of automata groups. In the particular case $\mathbb{Z} \wr (\mathbb{Z} \wr \mathbb{Z})$, we prove that it is a subgroup of a two-letters automata group; this solves Problem 15.19 - (b) of the Kourovka Notebook proposed by A. M. Brunner and S. Sidki in 2000 \cite{BruSid3, NK}.  
\end{abstract}

\section{Introduction}
An element $\alpha = (\alpha_1, \dots, \alpha_m)\sigma(\alpha)$ of the automorphism group $\mathcal{A}_{m} \simeq \mathcal{A}_{m} \wr S_{m}$ of the one-rooted $m$-regular tree $\mathcal{T}_{m}$ is \textit{finite-state} if the set of states $Q(\alpha) = \{\alpha\} \cup \cup_{i = 1}^{m} Q(\alpha_{i})$ of $\alpha$ is finite.  The set of all finite-state automorfisms forms a countable subgroup $\mathcal{F}_m$ of $\mathcal{A}_m$, see \cite{BruSid1}. A subgroup of $\mathcal{F}_{m}$ is called finite-state. A lot of interest in finite-state groups has been raised in the past decades, and important groups have finite-state representations, e.g., the linear general group GL$(n, \mathbb{Z})$ (see \cite{BruSid2}), wreath product of groups (see \cite{S}), the free abelian group of countable rank (see \cite{BruSid1}), the free metabelian group $\mathbb{M}_{r}$ of finite rank $r$, and the lamplighter groups $\mathbb{Z}^{r} \wr \mathbb{Z}^{s}$ (see \cite{BruSid3}). Brunner and Sidki in \cite{BruSid3} asked if the group $\mathbb{Z} \wr (\mathbb{Z} \wr \mathbb{Z})$ has a finite-state representation - the main objective of these notes is to answer this question affirmatively.

Another interesting class of groups is the family of \textit{self-similar} groups. A subgroup $G$ of $\mathcal{A}_{m}$ is self-similar of degree $m$ if, for any $\alpha$ in $G$, the set $Q(\alpha)$ is a subset of $G$. A self-similar group can be directly given as a subgroup of $\mathcal{A}_{m}$ - as is the case of the infinite torsion Grigorchuk group (see \cite{G}), the infinite torsion Gupta-Sidki group (see \cite{GuS}), and the Hanoi towers group (see \cite{BSZ, RS}) - or
constructed to act on $\mathcal{T}_{m}$ by way of finitely many virtual endomorphisms - as is the case of the free abelian
group of countable rank (see \cite{BarSid, DSS}), BFC-groups (see \cite{D}), finitely generated nilpotent groups (see \cite{BDS, BerSid}), metabelian groups (see \cite{DS, KS}), affine linear groups (see \cite{BruSid3}), and arithmetic groups (see \cite{K}). A finitely generated subgroup $G$ of $\mathcal{F}_{m}$ is said to be an \textit{automata group} if $G$ is self-similar; in other words, $G$ is an automata group if for any $\alpha$ in $G$, the set $Q(\alpha)$ is a finite subset of $G$. The most well-known self-similar groups given directly as a subgroup of $\mathcal{A}_{m}$ are automata groups. It is worth noting that any finitely genereted subgroup $G$ of $\mathcal{F}_{m}$ is a subgroup of an automata group. Indeed, if $G = \langle x_{1}, \dots, x_{n} \rangle$, then $Q = \langle x_{1}, \dots, x_{n}, Q(x_{1}),\dots, Q(x_{n}) \rangle$ is an automata group, and $G \leqslant Q$.

We use the approach of virtual endomorphisms to establish that the group $A \wr G$ is a subgroup of a self-similar group, where $A$ is a finitely generated abelian group and $G$ is a subgroup of a self-similar group. A virtual endomorphism of an abstract group $G$ is a homomorphism $f: H \rightarrow G$ from a subgroup $H$ of finite index in $G$. Consider a group $G$ and 
$${\bf H} = (H_{i} \leqslant G \mid [G : H_{i}] = m_{i}, \, (1 \leq i \leq s)),$$
$${\bf m} = (m_1, \dots , m_s), \, m_1 + \dots + m_{s} = m,$$
$${\bf F} = (f_i: H_i \rightarrow G \, \text{virtual endomorphism} \mid (1 \leq i \leq s)).$$
Then the $G$-data $({\bf m}, {\bf H}, {\bf F})$ induces a self-similar representation $\varphi: G \rightarrow \mathcal{A}_{m}$ with kernel 
$${\bf F}-\text{core}({\bf H}) = \langle K\leqslant \cap _{i=1}^{s}H_{i}\mid K\vartriangleleft
		G,K^{f_{i}}\leqslant K,\forall i=1,\dots,s\rangle.$$
If the ${\bf F}$ - core$({\bf H})$ is trivial, the $G$-data is said to be \textit{simple} and the group $G$ is self-similar; see \cite{DSS} for more information. For $1 \leq i \leq s$, define the subgroup $H_{\omega_{i}}$ of $H_i$ by $H_{\omega_{i}} = \langle K \leqslant H_{i}\mid K^{f_{i}}\leqslant K \rangle,$
and define the \textit{parabolic subgroup} $H_{\omega}$ of $\cap _{i=1}^{s}H_{i}$ by
$$H_{\omega} = \langle K\leqslant \cap _{i=1}^{s}H_{i}\mid ,K^{f_{i}}\leqslant K, \forall i=1,\dots,s\rangle.$$
We have the following result about parabolic subgroups.

\begin{thmA}
A simple $G$-data $({\bf m}, {\bf H}, {\bf F})$ induces a simple $G$-data $({\bf l}, {\bf K}, {\bf E})$ with a trivial parabolic subgroup. Moreover, if the representation of $G$ induced by the $G$-data $({\bf m}, {\bf H}, {\bf F})$ is finite-state, then the representation induced by the $G$-data $({\bf l}, {\bf K}, {\bf E})$ is also finite-state.
\end{thmA}

If the $G$-data $({\bf m}, {\bf H}, {\bf F})$ is simple with ${\bf m} = (m_{1}, \dots, m_{s})$, then we say that $G$ is self-similar of \textit{orbit-type} $(m_{1},\dots, m_{s})$. In \cite{BarSid}, Bartoldi and Sidki showed that if $B$ is a finite abelian group and the $G$-data $((m), (H \leqslant G), (f: H \rightarrow G))$ is simple with parabolic subgroup $H_{\omega}$, then $\mathcal{G} = B^{(H_{\omega} \setminus G)} \rtimes G$ is self-similar with degree $|B|\cdot m$, where $(H_{\omega} \setminus G)$ is the set of left cosets $\{ H_{\omega}g \mid g \in G \}$. The following Theorem extends the Bartoldi-Sidki's result and fixes the Theorem E of \cite{DSS}.

\begin{thmB}
Let $A$ be a finitely generated abelian group and $G$ be a non-torsion self-similar group induced by the $G$-data $(\mathbf{m},\mathbf{H}, \mathbf{F})$. Let $H$ be the group $\cap_{i=1}^{s}H_i$. Then,
$$\mathcal{G}=A^{((H_{\omega_1} \setminus G) \times \cdots \times (H_{\omega_s} \setminus G))}\rtimes G^s$$ is a self-similar group.
If, in addition, $G$ is finite-state, then $\mathcal{G}$ is also finite-state.
\end{thmB}

As a consequence of Theorem B, we obtain that the groups $C_2 \wr (C_{2} \wr \mathbb{Z})$ and $\mathbb{Z} \wr (C_2 \wr \mathbb{Z})$ are finite-state, see Corollary \ref{C1}.

A subgroup $G$ of $\mathcal{A}_m$ is said to be \textit{transitive} if its action on the first level of the tree $\mathcal{T}_m$ is transitive. It was recently shown that the wreath product $\mathbb{Z} \wr \mathbb{Z}$ fails to have a faithful
transitive self-similar representation of degree $m$ for any $m$, see \cite{DS}. However, Dantas, Santos, and Sidki showed in \cite{DSS} that the group $\mathbb{Z} \wr \mathbb{Z}$ has a non-transitive self-similar finite-state representation of degree $3$. Define the $k$-iterated wreath product of $k$ copies of $\mathbb{Z}$ inductively by $W_{1}(\mathbb{Z}) = \mathbb{Z}$, $W_{k}(\mathbb{Z}) = \mathbb{Z} \wr W_{k-1}(\mathbb{Z})$ for $k > 1$. Using Theorem A, we have the following result.

\begin{thmC}
    Let $A$ be a finitely generated abelian group and $G$ be a non-torsion self-similar finite-state group, then $A \wr G$ is finite-state.
\end{thmC}

As a consequence of Theorem C, we obtain that the groups $C_{2} \wr (\mathbb{Z} \wr \mathbb{Z})$ and $W_{k}(\mathbb{Z})$ ( $k \geq 1$) are finite-state. For $k = 3$, Theorem C answers affirmatively the question of Brunner and Sidki in \cite{BruSid3}. The representation of $W_{3}(\mathbb{Z}) = \mathbb{Z} \wr (\mathbb{Z} \wr \mathbb{Z})$ in Theorem C is of degree $10$, meaning that the group $W_{3}(\mathbb{Z})$ in this representation is a subgroup of $\mathcal{F}_{10}$. We use a deflation on the degree of the tree, and we obtain the following.

\begin{thmD}
The group $W_{3}(\mathbb{Z}) = \mathbb{Z} \wr (\mathbb{Z} \wr \mathbb{Z})$ is a subgroup of $\mathcal{F}_2$.
\end{thmD}
By Theorem D, the group $W_{3}(\mathbb{Z}) = \mathbb{Z} \wr (\mathbb{Z} \wr \mathbb{Z})$ is a subgroup of a two-letters automata group, this answers affirmatively the Problem 15.19 - (b) of the Kourovka Notebook \cite{NK} proposed by Brunner and Sidki in \cite{BruSid3}.

\section{Preliminaries}

\subsection{The one-rooted $\mathbf{m}$-regular tree $\mathbf{\mathcal{T}_m}$ and its automorphisms}

For a vertex $v$ belonging to a rooted tree, its \textit{level} is defined as the length of the geodesic from the root ($\emptyset$) to $v$, denoted by $|v|$. A spherically homogeneous regular tree of valence $m$ ($m\geq 2$) will be denoted by $\mathcal{T}_m$. The vertices of  $\mathcal{T}_{m}$ will be indexed by strings
from the alphabet $Y=\{1,2,\dots,m\}$, ordered by $u<v$ provided the
string $v$ is a prefix of $u$. The tree $\mathcal{T}_{m}$ is also denoted as 
$\mathcal{T}\left( Y\right) $.


The automorphism group $\mathcal{A}_{m}$, or $\mathcal{A}\left( Y\right) $,
of $\mathcal{T}_{m}$ is isomorphic to the restricted wreath product
recursively defined as $\mathcal{A}_{m}=\mathcal{A}_{m}\wr S_{m}$, where $%
S_{m}$ is the symmetric group of degree $m$. An automorphism $\alpha $ of $%
\mathcal{T}_{m}$ has the form $\alpha =(\alpha _{1},\dots,\alpha _{m})\sigma
(\alpha )$, where the state $\alpha _{i}$ belongs to $\mathcal{A}_{m}$ and
where $\sigma :\mathcal{A}_{m}\rightarrow S_{m}$ is the permutational
representation of $\mathcal{A}_{m}$ on $Y$, the first level of the tree $%
\mathcal{T}_{m}$; the permutation $\sigma(\alpha)$ is called the \textit{activity of} $\alpha$. Successive developments of the automorphisms $\alpha _{i}$
produce $\alpha _{u}$ for all vertices $u$ of the tree. For $k\geq 1$, the
action of $\alpha $ on a string $y_{1}y_{2}\cdots y_{k}\in Y^{k}$ is as follows
\begin{equation*}
\alpha :y_{1}y_{2}\cdots y_{k}\mapsto \left( y_{1}\right) ^{\sigma (\alpha
)}\left( y_{2}\cdots y_{k}\right) ^{\alpha _{y_{1}}}\text{. }
\end{equation*}%
This implies that $\alpha $ induces an automorphism $\alpha _{\mathbf{k}}$
on the $m^{k}$-tree $\mathcal{T}\left( Y^{k}\right) $ and the above action
on $k$-strings gives us a group embedding $\mathcal{A}\left( Y\right)
\rightarrow \mathcal{A}\left( Y^{k}\right) $ which we call $k$-\textit{%
inflation}.

\subsection{Virtual endomorphisms and finite-state groups}

A \textit{virtual endomorphism} of an abstract group $G$ is a homomorphism $f:H \to G$ from a subgroup $H$ of finite index in $G$. Let $G$ be a group and consider 
\begin{equation*}
	\mathbf{H=} \left( H_{i} \leqslant G \mid \left[ G:H_{i}\right] =m_{i}\text{ }%
	\left( \text{ }1\leq i\leq s\right) \right) \text{,}
\end{equation*}%
\begin{equation*}
	\mathbf{m=}\left( m_{1},\dots,m_{s}\right) ,\text{ }m=m_{1}+\dots+m_{s},
\end{equation*}
\begin{equation*}
	\mathbf{F}=\left( f_{i}:H_{i}\rightarrow G \,\, \text{virtual endomorphisms}\mid 1\leq i\leq s\right) \text{;
	}
\end{equation*}
we will call  $\left( \mathbf{m},\mathbf{H,F}\right) $ a $G$-\textit{data} or \textit{data for} $G$. The $\mathbf{F}$-\textit{core} is the largest subgroup of $ \cap _{i=1}^{s}H_{i}$ which is normal in $G$ and $f_i$ - invariant for all $i=1,\dots,s$. In the case $\mathbf{F}$-core is trivial, we say that $\mathbf{F}$ is \textit{free}.

Let $(\bf{m}, \bf{H}, \bf{F})$ be a $G$-data and for each subgroup $H_i$ in $\mathbf{H}$, choose a right transversal $T_i=\{t_{i1},t_{i2},\cdots,t_{im_i}\}$ in $G$. Define the $i$-Schreier function $\theta_i: G \times T_i \to H_i$ by 
$$ \theta_i(g,t_{ij})=t_{ij}g(t_{ik})^{-1},\,\,\, \text{where}\,\,\, H_it_{ik}=H_it_{ij}g,$$
and let $\sigma$ be the induced permutation representation of $G$ on $T=T_1 \cup T_2 \cup \cdots \cup T_s$, which means that $j^\sigma=k$ if and only if $H_it_{ik}=H_it_{ij}g$ for some $i=1,\cdots,s$. The maps in $\mathbf{F}$ extends to a homomorphism

$$\varphi : G \to \mathcal{A}_{m}$$
defined by
$$g^{\varphi}=(\theta_i(g,t)^{f_i \varphi}\,\,|\,\,t\in T_i, 1\leq i \leq s)\,\,g^\sigma;$$
this is essentially the contents of the Kaloujnine-Krasner Theorem. 

An element $g\in G$ is called \textit{finite-state} if there exists a finite subset $U \subset G$ containing $g$ such that $(U^{T})^{\varphi} \subset U$, where $$U^{T} = \{x, \theta_i(x,t)^{f_i} \mid t \in T, x \in U\}.$$ 
A group $G$ is \textit{finite-state} if every element $g \in G$ is finite-state.


\subsection{Self-similarity of groups} For $\alpha = (\alpha_1, \dots, \alpha_m)\sigma(\alpha) \in \mathcal{A}\left( Y\right) $, the set of automorphisms $Q(\alpha)=\{\alpha_u \,\,| \,\,u \in Y^k,\,\, k\geq 0\} = \{\alpha\} \cup \cup_{i = 1}^{m} Q(\alpha_{i})$ is called the set of \textit{states} of $\alpha$. A subgroup 
$G$ of $\mathcal{A}_{m}$ is \textit{state-closed} (or \textit{self-similar}%
) if $Q(\alpha )$ is a subset of $G$ for all $\alpha $ in $G$. More
generally, a subgroup $G$ of $\mathcal{A}_{m}$ is said to be $k$th\textit{-level} 
\textit{state-closed} if, for all $\alpha \in G$, the states $\alpha _{u}$ belong to 
$G $ for all strings $u$ of length $k$; it follows that the $k$-inflation of 
$G $ is state-closed. A group which is finitely generated, state-closed and
finite-state is called an \textit{automata group}. \\

The following approach to produce 
self-similar groups was given in \cite{DSS}.


\begin{proposition} Given a group $G$, $m\geq 2$ and a $G$-data $\left( 
	\mathbf{m},\mathbf{H,F}\right) $. Then the data provides a self-similar
	representation of $G$ on the $m$-tree with kernel the $\mathbf{F}$-core, 
	\begin{equation*}
		\langle K\leqslant \cap _{i=1}^{s}H_{i}\mid K\vartriangleleft
		G,K^{f_{i}}\leqslant K,\forall i=1,\dots,s\rangle \text{.}
	\end{equation*}
\end{proposition}

We say that a $G$-data $(\bf{m}, \bf{H}, \bf{F})$ is \emph{simple} if the $\mathbf{F}$-core is trivial. In this case, $G$ is self-similar.

\begin{lemma}\label{2.2}
   Let $G$ be a self-similar group with respect to the data $(\mathbf{m},\mathbf{H},\mathbf{F})$. If $K_i$ is a subgroup of $H_i$ with $[H_i:K_i]=n_i$, then the data $(\mathbf{n},\mathbf{K},\mathbf{\overline{F}})$, where $$\mathbf{n}=(m_1n_1,\dots,m_sn_s)$$ $$\mathbf{K}=(K_1,\dots,K_s)$$ $$\mathbf{\overline{F}}=(f_1|_{K_1}, \dots, f_s|_{K_s}),$$ defines $G$ as self-similar. Moreover, if $G$ is finite-state in the first representation, then it is also finite-state with respect to the second representation.
\end{lemma}

\begin{proof}
Since $\mathbf{\overline{F}}$-core$(\mathbf{K}) \leq\mathbf{F}$-core$(\mathbf{H})$, the group $G$ is self-similar with respect to the data $(\mathbf{n},\mathbf{K},\mathbf{\overline{F}})$.

Let $T_i$ be a transversal of $H_i$ in $G$, and $L_i$ be a transversal of $K_i$ in $H_i$, for $i=1,\dots,s$. For any $i \in \{1, \dots, s\}$, define $L_iT_i$ as a transversal of $K_i$ in $G$.

Suppose that the transversals $T_1,\dots, T_s$ induce a  finite-state self-similar representation of $G$ with respect to $f_{1}, \dots, f_{s}$, respectively. Consider $\varphi$ and $\dot{\varphi}$ as representations of $G$ with respect to $(\mathbf{m},\mathbf{H},\mathbf{F})$ and $(\mathbf{n},\mathbf{K},\mathbf{\overline{F}})$, respectively.

$$\varphi: G \to \mathcal{A}_{m}$$
$$ \, \, \, \, \, \, \, \, \, \, \, \, \, \, \, \, \, \, \, \, \, \, \, \, \, \, \, \, \, \, \, \, \, \, \, \, \, \, \, \, \, \, \, \, \, \, \, \, \, \, \, \, \, \, \, \, \, \,  g \mapsto g^{{\varphi}}=((\theta_{T_{i}}(g,t)^{f_i})^{\varphi}\,\,|\,\,t\in T_i, 1\leq i \leq s)\,\,g^\sigma,$$

$$\dot{\varphi}: G \to \mathcal{A}_{n}, \text{ where } n = m_{1}n_{1} + \dots + m_{s}n_{s}$$
$$ \, \, \, \, \, \, \, \, \, \, \, \, \, \, \, \, \, \, \, \, \, \, \, \, g \mapsto g^{{\dot{\varphi}}}=((\theta_{L_{i}T_{i}}(g,lt)^{f_{ij}})^{\dot{\varphi}}\,\,|\,\,l \in L_{i}, t \in T_{i}, 1\leq i \leq s)\,\,g^{\dot{\sigma}}.$$

\,

Note that for each $g \in G$, we have
$$\theta_{L_{i}T_{i}}(g,l_bt_a)^{f_{i}} = \theta_{L_{i}}(\theta_{T_{i}}(g,t_{a}), l_{b})^{f_{i}} =$$ 
$$ = (l_bt_ag(t_{(a)\sigma(g)}^{-1})l_{(b)\dot{\sigma}(\theta(g, t_{a}))}^{-1})^{f_{i}} = l_b^{f_{i}}(t_ag(t_{(a)\sigma(g)}^{-1}))^{f_{i}}(l_{(b)\dot{\sigma}(\theta(g, t_{a}))}^{-1})^{f_{i}}.$$ 
Thus, there are finite subsets $U$ and $V$ of $G$ such that $g \in U$, $(U^{T})^{\varphi} \subset U$, $L_{i}^{f_{i}} \subset V, (1 \leq i \leq s)$, and $(V^T)^{\varphi} \subset V$. Then, $U_{0} = VUV$ is a finite subset of $G$ such that $g \in U_{0}$ and $(U_{0}^{LT})^{\dot{\varphi}} \subset U_{0}$.

\end{proof}


Let $G$ be a finite-state self-similar group with respect the data
$$((\underbrace{m,\dots,m}_\text{$s$ times}), (\underbrace{H, \dots, H}_\text{$s$ times}), (f_{1}, \dots, f_{s})).$$ Consider the $G^s$-datas $\left( \mathbf{
m_{1},H_{1},F_{1}} \right)$ and $\left( \mathbf{
m_{2},H_{2},F_{2}}\right)$ given respectively by
$$\left( \mathbf{
m_{1},H_{1},F_{1}} \right) = ((\underbrace{m^s, \dots, m^s}_\text{$s$ times}), (\underbrace{H^s, \dots, H^s}_\text{$s$ times}), (\rho_{1}, \dots, \rho_{s}))$$
where $\rho_{(1)\sigma^i}: H^s \rightarrow G^s$ extends the map
$$\rho_{(1)\sigma^i}: (h_1,\dots,h_s) \mapsto (h_1^{f_{(1)\sigma^i}},\dots, h_s^{f_{(s)\sigma^i}}),$$
$\sigma = (1 \, 2 \, \cdots s) \in Sym(\{1, 2, \dots, s\})$, and $i \in \{1, 2, \dots, s\}$;

$$\left( \mathbf{
m_{2},H_{2},F_{2}} \right) = ((m^s, 1), (H^s, G^s), (\rho, \tau)),$$
where $\rho: H^s \rightarrow G^s$ and $\tau: G^s \rightarrow G^s$ extends, respectively, the maps
$$\rho: (h_1,\dots,h_s) \mapsto (h_1^{f_1},\dots, h_s^{f_s}),$$
$$\tau: (g_1,\dots,g_s) \mapsto (g_s, g_1\dots, g_{s-1}).$$

\begin{lemma}
    In the above hypotheses, both the $G^s$-datas $\left( \mathbf{
m_{1},H_{1},F_{1}} \right)$ and $\left( \mathbf{
m_{2},H_{2},F_{2}}\right)$ yilds $G^s$
    as finite-state self-similar.
\end{lemma}

\begin{proof}

    Let us first prove that $G^s$ is a self-similar group with respect to $\left( \mathbf{
m_{1},H_{1},F_{1}} \right)$. Let $K\leqslant H^s$, $K \vartriangleleft G^s$ and $K^{\rho_i}\leqslant K$, $i=1,\dots, s$.  Define $\pi_i: H^s\to H$ as the projection of $H^s$ on its $i$-th coordinate. Then
$$K_i:=\pi(K)=\{h \in H\,\,|\,\, (x_1,\dots,x_{i-1},h,x_{i+1},\dots, x_s) \in K\}$$
is a subgroup of $H$ that is normal in $G$. Since $K$ is $\rho_i$ - invariant for all $i=1,\dots, s$, $K_i$ is  $f_i$ - invariant for all $i=1,\dots, s$. By the self-similarity of $G$, we conclude that $K_i=\{e\}$ for all $i=1,\dots,s$. As $K=K_1\times \cdots \times K_s$, we obtain that $G^s$ is self-similar.  

Now, we will prove that if $G$ finite-state, then $G^s$ is also finite-state. Consider $\varphi: G \to \mathcal{A}_{sm}$ a finite-state representation of $G$ with respect to the transversals $T_i=\{t_{i1},\cdots, t_{im}\} \,(i=1,\cdots,s)$ of $H$ in $G$. Then,
$$g^{{\varphi}}=((\theta_i(g,t)^{f_i})^{\varphi}\,\,|\,\,t\in T_i, 1\leq i \leq s)\,\,g^\sigma.$$
Note that $T_1\times \cdots \times T_s$ is a transversal of $H^s$ in $G^s$. With respect to this transversal, we have that $G^s$ has the following representation $$\dot{\varphi}: G^s \to \mathcal{A}_{sm^s}$$
$$(g_1,\cdots,g_s)^{{\dot{\varphi}}}=((\theta_i(g_1,t_{ij_1})^{f_{(1)^{\sigma^i}}})^{\varphi},\cdots, (\theta_i(g_s,t_{ij_s})^{f_{(s)^{\sigma^i}}})^{\varphi}\,\,|\,\,$$ $${t_{ij_k}}\in T_{i}, 1\leq j_1,\cdots,j_s\leq m, 1\leq i \leq s)\,\,g_1^\sigma\cdots g_s^\sigma.$$
Thus, each state of $G^s$ in the representation $\dot{\varphi}$ occurs as a state of $G$ with respect to $\varphi$, implying that $G^s$ is finite-state. 

The proof that $G^s$ is a finite-state self-similar group with respect to the data $\left( \mathbf{
m_{2},H_{2},F_{2}}\right)$ is similar. 

\end{proof}

\subsection{\noindent \textbf{General concatenation}. }

Let $G_{1}=A_{1}\rtimes U$, $G_{2}=A_{2}\rtimes U$ and $G=\left( A_{1}\oplus
A_{2}\right) \rtimes U$. For $i=1,2$, define the $G_{i}$-data $\left( \mathbf{
m_{i},H_{i},F_{i}}\right)$, where $\mathbf{m_{i}=}\left(
m_{i1},\dots,m_{is_{i}}\right) $, $\mathbf{H_{i}}=\{H_{i1},\dots,H_{is_{i}}\}$
and $\mathbf{F_{i}}=\{f_{i1},\dots,f_{is_{i}}\}$. Furthermore, define the data 
$G$-data $\left( \mathbf{m,H,F}\right)$, where $\mathbf{m}$ is the
concatenation $\left( \mathbf{m_{1},m_{2}}\right) $, $\mathbf{H}$ is the
concatenation 
\begin{eqnarray*}
\left( \tilde{H}_{1j}={A_{2}}
\cdot H_{1j}\text{ }\left( 1\leq j\leq s_{1}\right), \, \tilde{H}_{2k}= {A_{1}} \cdot H_{2k}\text{ }%
\left( 1\leq k\leq s_{2}\right) \right) \text{,}
\end{eqnarray*}%
and
\begin{equation*}
\mathbf{F=}(\tilde{f}_{11},\dots,\tilde{f}_{1s_{1}},\tilde{f}_{21},\dots,\tilde{%
f}_{2s_{2}})
\end{equation*}%
where $\tilde{f}_{1j}:\tilde{H}_{1j}\rightarrow G, \,\, 1\leq j \leq s_1,$ is defined by 
\begin{equation*}
\tilde{f}_{1j}:ah\mapsto h^{f_{1j}}\text{, for }a\in {A_{2}},h\in H_{1j}
\end{equation*}%
and $\tilde{f}_{2k}:\tilde{H}_{2k}\rightarrow G, \,\, 1\leq k \leq s_2,$ is defined by 
\begin{equation*}
\tilde{f}_{2k}:ah\mapsto h^{f_{2k}}\text{, for }a\in {A_{1}}\text{, }%
h\in H_{2k}\text{.}
\end{equation*}

\begin{proposition}\label{P3.2} In the given seetup, for $i=1,2$, let $G_{i}$ have its state-closed representation with respect
to $\left( \mathbf{m_{i},H_{i},F_{i}}\right) $. Likewise, let $G$ have its state-closed representation with respect
to $\left( \mathbf{m,H,F}\right)$. Suppose the above state-closed representations of $G_{1}$ and $G_{2}$
are faithful. Then so is the corresponding state-closed representation of $G$. If, in addition, $G_1$ and $G_2$ are finite-state, then $G$ is also finite-state.
\end{proposition}

The above proposition was established in \cite{DSS}, see Proposition 5.2.

\section{The Parabolic Subgroup}

Let $G$ be a group, $m\geq 2$ and $\left(\mathbf{m},\mathbf{H,F}\right)$ be a $G$-data. If $L$ is a subset of $G$, denote $L^{f_{i}^{-1}}$ as the subset $\{h \in H_{i} \mid h^{f_i} \in L \}$ of $H_{i}$. We say that
the subgroup $H_{\omega} = \bigcap_{k=0}^{\infty} W_{k}$, where
$$ W_{0} = G, W_{1} = \bigcap_{i=1}^{s} W_{0}^{f_{i}^{-1}},  W_{2} = \bigcap_{i=1}^{s} W_{1}^{f_{i}^{-1}}, \dots, W_{k} = \bigcap_{i=1}^{s} W_{k-1}^{f_{i}^{-1}},$$
is the \textit{parabolic subgroup} of $G$. It is immediate that
$$H_{\omega} = \langle K \leq \cap_{i=1}^{s} H_{i} \mid K^{f_{i}} \leq K, \text{ for all }i = 1, \dots, s \rangle.$$

\begin{thmA} \label{T2.5}
A simple $G$-data $({\bf m}, {\bf H}, {\bf F})$ induces a simple $G$-data $({\bf l}, {\bf K}, {\bf E})$ with trivial parabolic subgroup. Moreover, if the representation of $G$ induced by the $G$-data $({\bf m}, {\bf H}, {\bf F})$ is finite-state, then the representation induced by the $G$-data $({\bf l}, {\bf K}, {\bf E})$ is also.
\end{thmA}

\begin{proof}
Consider
$$({\bf m}, {\bf H}, {\bf F}) = ((m_1,\dots,m_s), (H_1,\dots, H_s), (f_1,\dots,f_s)),$$ and $H_{G}$ the core of the subgroup $\cap_{i=1}^{s} H_i$ in $G$. By Lemma \ref{2.2}, the group $G$ is self-similar with respect to the data
$$({\bf m_{1}}, {\bf H_{1}}, {\bf F_{1}}) = ((\underbrace{r,\dots,r}_\text{$s$ times}),(\underbrace{H_{G}, \dots, H_{G}}_\text{$s$ times}),(f_{1}, f_{2}, \dots,f_{s})).$$

Let $T_1, \dots, T_s$ be transversals of $H_{G}$ in $G$, where $T_{i} = \{t_{i1} = e, t_{i2}, ..., t_{ir} \}$. Define the endomorphism
\begin{eqnarray*}
f_{ij}: H_{G}&\rightarrow& G\\
h &\mapsto& h^{t_{ij}f_i} = (t_{ij}^{-1} h t_{ij})^{f_{i}}
\end{eqnarray*}
for each $i=1, 2, \dots, s$ and each $j = 1, 2, \dots, r$, and define the $G$-data $({\bf l}, {\bf K}, {\bf E})$ by
$$
((\underbrace{r,\dots,r}_\text{$sr$ times}),(\underbrace{H_{G}, \dots, H_{G}}_\text{$sr$ times}),(f_{11}, f_{12}, \dots,f_{1r},\dots, f_{s1}, f_{s2}, \dots, f_{sr})).    
$$
The parabolic subgroup $K_\omega$ of the data $({\bf l}, {\bf K}, {\bf E})$ is
$$K_\omega = \langle K \leq H_{G} \mid  K^{f_{ij}}\leq K, \, \forall (i, j) \in \{1,\dots,s\} \times \{1, \dots, r\} \rangle.$$
Set $x \in K_{\omega}$ and $g \in G$. There exist $t_{j} \in T$ and $h \in H_{G}$ such that $g = ht_{ij}$. For each $i = 1, \dots, s$ we have $$(x^g)^{f_{i}}  = (x^{ht_{ij}})^{f_{i}}  = x^{hf_{ij}} = (x^{f_{ij}})^{h^{fij}} = y^{g_{1}},$$
where $x^{f_{ij}} = y \in K_{\omega}$ and $h^{f_{ij}} = g_{1}$. Thus $K_\omega^{G}$ is $f_{i}$-invariant for each $i=1,\dots,s$ and since the date
$({\bf m}, {\bf H}, {\bf F})$
is simple, it follows that  $K_{\omega} = 1$. Therefore, the data $({\bf l}, {\bf K}, {\bf E})$ induces a self-similar representation of $G$ with trivial parabolic subgroup.

Now, by Lemma \ref{2.2}, we can assume that the transversals $T_1, \dots, T_s$ of $H_{G}$ in $G$ induce a finite-state self-similar representation of $G$ with respect to $f_{1}, \dots, f_{s}$, respectively. For any $i \in \{1, \dots, s\}$, define $T_{ij} = x_j T_i$, where $x_{j} \in T_{i}$, and $j \in \{1, \dots, r\}$.  We claim that the transversals $T_{11}, \dots, T_{1r}, \dots, T_{s1}, \dots, T_{sr}$ of $H_{G}$ in $G$ induce a finite-state self-similar representation of $G$ with respect to $f_{11}, \dots,f_{1r}, \dots, f_{s1}, \dots, f_{sr}$, respectively. Let 
$$\varphi: G \to \mathcal{A}_{sr}$$
$$ \, \, \, \, \, \, \, \, \, \, \, \, \, \, \, \, \, \, \, \, \, \, \, \, \, \, \, \, \, \, \, \, \, \, \, \, \, \, \, \, \, \, \, \, \, \, \, \, \, \, \, \, \, \, \, \, \, \,  g \mapsto g^{{\varphi}}=((\theta_i(g,t)^{f_i})^{\varphi}\,\,|\,\,t\in T_i, 1\leq i \leq s)\,\,g^\sigma$$
and
$$\dot{\varphi}: G \to \mathcal{A}_{sr^2}$$
$$ \, \, \, \, \, \, \, \, \, \, \, \, \, \, \, \, \, \, \, \, \, \, \, \, g \mapsto g^{{\dot{\varphi}}}=((\theta_{ij}(g,t)^{f_{ij}})^{\dot{\varphi}}\,\,|\,\,t\in T_{ij}, 1\leq i \leq s, 1\leq j \leq r )\,\,g^{\dot{\sigma}}$$
the representations of $G$ with respect to the $G$-datas $({\bf m_{1}}, {\bf H_{1}}, {\bf F_{1}})$ and $({\bf l}, {\bf K}, {\bf E})$, respectively. Note that if $t \in T_{ij}$, then $t = x_{j}t_{l}$ for some $t_{l} \in T_{i}$, and so
$$\theta_{ij}(g, t) = \theta_{ij}(g, x_{j}t_{l}) = x_{j}t_{l}g(x_{j}t_{(l)\sigma_{g}})^{-1} = (t_{l}gt_{(l)\sigma_{g}}^{-1})^{x_{j}^{-1}}.$$
Indeed, there exists $h \in H_{G}$ such that $t_{l} g = h t_{(l)\sigma(g)}$, and so
$$x_{j} t_{l} g = x_{j} h t_{(l)} = h^{x_{j}^{-1}} x_{j} t_{(l)\sigma(g)} \in H_{G} x_{j} t_{(l)\sigma(g)}.$$
Thus, the permutation of $T_{i}$ induced by $g$ is the same as the permutation of $T_{ij}$ induced by $g$, and so


\begin{equation*}
			\begin{split}
	g^{{\dot{\varphi}}} &= ((\theta_{ij}(g,t)^{f_{ij}})^{\dot{\varphi}}\,\,|\,\,t\in T_{ij}, 1\leq i \leq s, 1\leq j \leq r )\,\,g^{\dot{\sigma}}  \\ & =((((t_{l}gt_{(l)\sigma_{g}}^{-1})^{x_{j}^{-1}})^{x_{j}f_{i}})^{\dot{\varphi}}\,\,|\,\, x_{j}t_{l} \in T_{ij}, 1\leq i \leq s, 1\leq j \leq r )\,\,g^{\dot{\sigma}}\\& =(((t_{l}gt_{(l)\sigma_{g}}^{-1})^{f_{i}})^{\dot{\varphi}}\,\,|\,\, x_{j}t_{l} \in T_{ij}, 1\leq i \leq s, 1\leq j \leq r )\,\,g^{\dot{\sigma}} \\ &= ((\theta_i(g,t_{l})^{f_i})^{\dot{\varphi}}\,\,|\,\, x_{j}t_{l} \in T_{ij}, 1\leq i \leq s, 1\leq j \leq r )\,\,g^{\dot{\sigma}}.\\
		\end{split}
		\end{equation*}
The result follows.

\end{proof}

\begin{ex} \label{ex3.2}
Consider the group $G = \mathbb{Z} \wr \mathbb{Z} = \langle y \rangle \wr \langle x \rangle$. The $G$-data 
$$({\bf m}, {\bf H}, {\bf F}) = ((2, 1), (H, G), (f_{1}, f_{2}))$$ 
 where $H = \langle y \rangle^{\langle x \rangle}\langle x^2 \rangle$ and the homomophisms $f_{1}: H \rightarrow G$ and $f_{2}: G \rightarrow G$ extend respectively the maps
 $$y^{x^{2n}} \mapsto y^{x^{n}}, \, y^{x^{2n+1}} \mapsto 1, \, x^{2n} \mapsto x^{n}, n \in \mathbb{Z},$$
 and
 $$y \mapsto x, \, x \mapsto 1,$$
 induces the self-similar finite-state faithful representation
 $$G \simeq \langle \gamma = (\gamma, e, \alpha), \alpha = (e, \alpha, e)(1 \, 2) \rangle \leq \mathcal{A}_{3},$$
 see \cite[Theorem C]{DSS}. Note that the parabolic subgroup $H_{\omega}$ of the $G$-data $({\bf m}, {\bf H}, {\bf F})$ is not trivial since
 $$1 = (y^{x^3 - x})^{f_{1}} = (y^{x^3 - x})^{f_{2}}.$$
 
 As $H$ is normal in $G$, $T = \{1, x\}$ is a transversal of $H$ in $G$, and $f_{2}|_{H} = xf_{2}|_{H}$, applying the Theorem 3.1 we have that the $G$-data 
 $$({\bf l}, {\bf K}, {\bf E}) = ((2, 2, 2), (K_{1} = H, K_{2} = H, K_{3} = H), (f_{1}, xf_{1}, f_{2}|_{H}))$$
 has trivial parabolic subgroup $K_{\omega}$ and induces the faithful self-similar finite-state representation
 $$G \simeq \langle \gamma_{1} = (\gamma_{1}, e, e, \gamma_{1}, \alpha_{1}, \alpha_{1}), \alpha_{1} = (e, \alpha_{1}, e, \alpha_{1}, e, e)(1 \, 2)(3 \, 4) \rangle \leq \mathcal{A}_{6}.$$
\end{ex}

\,

Let $n$ be a positive integer and $\sigma = (1 \, 2 \, \dots \, s) \in Sym(\{1, 2, \dots, s\})$. For $\delta = (\delta_{1}, \dots, \delta_{n}) \in \{1, \dots, s\}^{n}$, consider $W_{\delta, 0} = G$ and  

$$W_{\delta,1} = \bigcap_{i = 1}^{s} W_{\delta, 0}^{f_{(i)(\sigma^{\delta_{n}})}^{-1}f_{(i)(\sigma^{\delta_{n-1}})}^{-1}...f_{(i)(\sigma^{\delta_{1}})}^{-1}},$$

$$W_{\delta,2} = \bigcap_{i = 1}^{s} (W_{\delta,1})^{f_{(i)(\sigma^{\delta_{n}})}^{-1}f_{(i)(\sigma^{\delta_{n-1}})}^{-1}...f_{(i)(\sigma^{\delta_{1}})}^{-1}},$$

$$W_{\delta,3} = \bigcap_{i = 1}^{s} (W_{\delta,2})^{f_{(i)(\sigma^{\delta_{n}})}^{-1}f_{(i)(\sigma^{\delta_{n-1}})}^{-1}...f_{(i)(\sigma^{\delta_{1}})}^{-1}},$$

$$\dots$$

$$W_{\delta,k} = \bigcap_{i = 1}^{s} (W_{\delta,k-1})^{f_{(i)(\sigma^{\delta_{n}})}^{-1}f_{(i)(\sigma^{\delta_{n-1}})}^{-1}...f_{(i)(\sigma^{\delta_{1}})}^{-1}}.$$
Note that $W_{\delta, \omega} = \cap_{k = 0}^{\infty} W_{\delta, k}$ is equal to $H_{\omega}$. In fact, we clearly have $H_{\omega} \leq W_{\delta, \omega}$. On the other hand, $W_{\delta, 0} = W_{0}$ and
$$W_{\delta,1} = \bigcap_{i = 1}^{s} W_{\delta, 0}^{f_{(i)(\sigma^{\delta_{n}})}^{-1}f_{(i)(\sigma^{\delta_{n-1}})}^{-1}...f_{(i)(\sigma^{\delta_{1}})}^{-1}} \leq \bigcap_{i = 1}^{s} W_{0}^{f_{(i)(\sigma^{\delta_{1}})}^{-1}} = \bigcap_{i = 1}^{s} W_{0}^{f_{i}^{-1}} = W_{1},$$

$$W_{\delta,2} = \bigcap_{i = 1}^{s} (W_{\delta,1})^{f_{(i)(\sigma^{\delta_{n}})}^{-1}f_{(i)(\sigma^{\delta_{n-1}})}^{-1}...f_{(i)(\sigma^{\delta_{1}})}^{-1}} \leq \bigcap_{i = 1}^{s} (W_{\delta,1})^{f_{(i)(\sigma^{\delta_{1}})}^{-1}} =$$
$$= \bigcap_{i = 1}^{s} (W_{\delta,1})^{f_{i}^{-1}} \leq \bigcap_{i = 1}^{s}W_{1}^{f_{i}^{-1}} = W_{2},$$

$$ \dots $$

$$W_{\delta,k} = \bigcap_{i = 1}^{s} (W_{\delta,k-1})^{f_{(i)(\sigma^{\delta_{n}})}^{-1}f_{(i)(\sigma^{\delta_{n-1}})}^{-1}...f_{(i)(\sigma^{\delta_{1}})}^{-1}} \leq \bigcap_{i = 1}^{s} (W_{\delta,k-1})^{f_{(i)(\sigma^{\delta_{1}})}^{-1}} = $$
$$= \bigcap_{i = 1}^{s} (W_{\delta,k-1})^{f_{i}^{-1}} \leq \bigcap_{i = 1}^{s} W_{k-1}^{f_{i}^{-1}} = W_{k}.$$
Thus $W_{\delta, \omega} = \cap_{k = 0}^{\infty} W_{\delta, k} \leq \bigcap_{k=0}^{\infty} W_{k} = H_{\omega}$, and the result follows.

For each positive integer $n$ and each $n$-tuple $\delta = (\delta_{1}, \dots, \delta_{n})$ with $0 \leq \delta_{1}, \dots, \delta_{n} \leq s - 1$, consider the map
$$\lambda_{n, \delta}: \Delta((H_{\omega} \setminus W_{\delta,1})^{s}) \rightarrow (H_{\omega} \setminus G)^{s}$$
$$ \, \, \, \, \, \, \, \, \, \, \, \, \, \, \, \, \, \, \, \, \, \, \, \, \, \, \, \, \, \, \, \, \, \, \, \, \, \, \, \, \, \, \, \, \, \, \, \, \, \, \, \, \, \, \, \, \, \, \, \, \, \, \, \, \, \, \, \, \, \, \, \, \, \, \, \, \, \,(H_{\omega}h)_{i=1}^{s} \mapsto (H_{\omega}h^{f_{(i)(\sigma^{\delta_{1}})}f_{(i)(\sigma^{\delta_{2}})}...f_{(i)(\sigma^{\delta_{n}})}})_{i=1}^{s}$$

\,

\noindent where $\Delta((H_{\omega} \setminus W_{\delta, 1})^{s}) = \{(H_{\omega}h)_{i = 1}^{s} \mid h \in W_{\delta, 1}\}$.

\begin{proposition} \label{3.3}
For any positive integer $n$  and any $n$-tuple $\delta$, the map $\lambda_{n, \delta}$ is injective.
\end{proposition}

\begin{proof}
If there exist $h, k \in W_{n, 1}$ such that 
$((H_{\omega}h)_{i=1}^{s})^{\lambda_{n, \delta}} = ((H_{\omega}k)_{i=1}^{s})^{\lambda_{n, \delta}}$, then $$H_{\omega} h^{f_{(i)(\sigma^{\delta_{1}})}f_{(i)(\sigma^{\delta_{2}})}...f_{(i)(\sigma^{\delta_{n}})}} = H_{\omega} k^{f_{(i)(\sigma^{\delta_{1}})}f_{(i)(\sigma^{\delta_{2}})}...f_{(i)(\sigma^{\delta_{n}})}}$$ for any $i \in \{1, \dots, s\}$. So 
$$(hk^{-1})^{f_{(i)(\sigma^{\delta_{1}})}f_{(i)(\sigma^{\delta_{2}})}...f_{(i)(\sigma^{\delta_{n}})}} \in H_{\omega}$$ 
for any $i \in \{1, \dots, s\}$ and since
$$H_{\omega} = \langle K \leq W_{n, 0} \mid K^{f_{(i)(\sigma^{\delta_{1}})}f_{(i)(\sigma^{\delta_{2}})}...f_{(i)(\sigma^{\delta_{n}})}} \leq K, i = 1, \dots, s\rangle$$
we have $H_{\omega}h = H_{\omega}k$. 
\end{proof}

\section{State-closed semidirect product of type $A^{((H_{\omega_{1}} \setminus G) \times \dots \times (H_{\omega_{s}} \setminus G))} \rtimes G^s$ with $A$ abelian}

\noindent Let $G$ be a state-closed group with respect the data $(\mathbf{m}, \mathbf{H}, \mathbf{F})$, where $\mathbf{m} = (m_{1}, \dots,
m_{s})$, with $m_{1} \geq 2, \dots, m_{s} \geq 2$. Suppose that $G$ is finite-state with respect to the data
$$((m, \dots, m), (H = \cap_{i = 1}^{s} H_{i}, \dots, H = \cap_{i = 1}^{s} H_{i}), (f_{1}, \dots, f_{s})).$$
For each $i=1,\dots,s$ define $H_{i0}=G$, $H_{ij}= (H_{i (j - 1)})^{f_i^{-1}}$ ($j>0$), and $H_{\omega_i}= \cap_{j\geq 0}H_{ij}$. With this notation, we have:

\begin{thm}\label{T3}
Let $G$ be self-similar group of
orbit-type $(m, \dots,
m)$, and let $m = [G : H]$.
\begin{enumerate}
    \item [(i)] If $B$ is a finite abelian group, then $$\mathcal{G}=B^{((H_{\omega_1} \setminus G) \times \cdots \times (H_{\omega_s} \setminus G))}\rtimes G^s$$
    is self-similar of orbit-type $(\underbrace{|B|\cdot m^s, \cdots, |B|\cdot m^s}_{s \,\, \text{times}})$; \\
    
    \item[(ii)] If $G$ is a non-torsion group, then $$\mathcal{G}=\mathbb{Z}^{((H_{\omega_1} \setminus G) \times \cdots \times (H_{\omega_s} \setminus G))}\rtimes G^s$$
    is self-similar of orbit-type $(\underbrace{m^s, \dots , m^s}_{s\,\, \text{times}}, 1)$.
\end{enumerate}
Moreover, if $G$ is finite-state, then $\mathcal{G}$ is also.

\end{thm}


\begin{proof}
Let $H$ be the subgroup $\bigcap_{i = 1}^{s} H_{i}$ of $G$. Consider the map  defined on the $s$-tuples of cosets
\begin{eqnarray*}
\lambda&:&  \prod_{i=1}^{s}\left(H_{\omega_i}\setminus \bigcap_{i = 1}^{s} H_{i}\right) \rightarrow \prod_{i=1}^{s}(H_{\omega_i}\setminus G)
\end{eqnarray*}%

\begin{equation*}
\, \, \, \, \, \, \, \, \, \, \, \, \, \, \, \, \, \, \, \, \, \, \, \, \, \, \, \, \, \,(H_{\omega_1}h_1,\dots,H_{\omega_s}h_s)\mapsto  (H_{\omega_1}h_1^{f_{1}},\dots,H_{\omega_s}h_s^{f_{s}}).
\end{equation*}%
If $(H_{\omega_i}h_i)^{\lambda}=(H_{\omega_i}h_i')^{\lambda}$, then $(h_i(h_i')^{-1})^{f_{i}}\in  H_{\omega_{i}}$ for each $1\leq i \leq s$; therefore $\lambda$ is injective.

To prove $(i)$ let $\mathcal{H}$ be the subgroup
$$\langle [B^{((H_{\omega_1} \setminus G) \times \cdots \times (H_{\omega_s} \setminus G))}, G^{s}] \rangle \rtimes H^s.$$
Note that the index $[\mathcal{G}:\mathcal{H}]$ is $|B|\cdot m^s$. 

Let $\sigma = (1 \, 2 \, \cdots s) \in Sym(\{1, 2, \dots, s\})$. For each $i \in \{1, 2, \dots, s\}$ define the homomorphism $\rho_{(1)\sigma^i}: \mathcal{H} \rightarrow \mathcal{G}$ that extend the map

$$\rho_{(1)\sigma^i}: [b^{(H_{\omega_1}, \cdots ,H_{\omega_s})}, (g_{1}, \cdots, g_{s})] = b^{- (H_{\omega_1}, \cdots ,H_{\omega_s}) + (H_{\omega_1} g_{1}, \cdots ,H_{\omega_s} g_{s})} \mapsto
$$
$$\mapsto \begin{cases}
       b^{- (H_{\omega_1}, \cdots ,H_{\omega_s}) + (H_{\omega_1} g_{1}^{f_{(1)\sigma^i}}, \cdots ,H_{\omega_s} g_{s}^{f_{(s)\sigma^i}})}, \textrm{ if } (g_1,\dots,g_s)\in H^s\\
       b^{- (H_{\omega_1}, \cdots ,H_{\omega_s})},  \textrm{ if } (g_1,\dots,g_s)\notin H^s
    \end{cases}$$
$$(h_1,\dots,h_s) \mapsto (h_1^{f_{(1)\sigma^i}},\dots, h_s^{f_{(s)\sigma^i}}), \, (h_1,\dots,h_s) \in H^s$$

We claim that the state-closed representation of $\mathcal{G}$ defined by the data $((|B|\cdot m^s, \dots, |B|\cdot m^s), (\mathcal{H}, \dots, \mathcal{H}), \{\rho_{1}, \dots, \rho_{s} \})$ is faithful.

Consider $K \leq \mathcal{H}$ with $K \vartriangleleft \mathcal{G}$ and $K^{\rho_{i}} \leq K$, $i = 1, 2, \dots, s$. Firstly, since $G^s$ is state-closed with relation to the tuple 
$$((m^s, \dots, m^s), (H, \dots, H), \{\rho_{1}|_{H}, \dots, \rho_{s}|_{H} \}),$$ 
then $K \leq \langle [B^{((H_{\omega_1} \setminus G) \times \cdots \times (H_{\omega_s} \setminus G))}, G^{s}] \rangle$. Let
$$y=b^{n_1(H_{\omega_1}h_{11},\dots,H_{\omega_s}h_{1s}) + n_2(H_{\omega_1}h_{21},\dots,H_{\omega_s}h_{2s}) +\dots + n_r(H_{\omega_1} h_{r1},\dots,H_{\omega_s} h_{rs})}$$
be an element of $K \setminus 1$ with $|y| = r > 1$ minimal. By normality of $K$, we can assume that $h_{11} = \cdots = h_{1s} = 1$. Since $\lambda$ is injective and $r$ is minimal, it follows that $y^{\rho_{1}^t} \neq 1$, and 
$$|y^{\rho_{1}^{n}}| = |b^{n_1(H_\omega,\dots,H_\omega)+ n_2(H_\omega k_{2, 1},\dots,H_\omega k_{2, s}) + \dots + n_r(H_\omega k_{r, 1},\dots,H_\omega k_{r,s})}| = r$$
for any $n = 0, 1, 2, \dots$. Then 
$$(h_{2 1}, \dots, h_{2 s}), \dots, (h_{r 1}, \dots, h_{r s}) \in H_{\omega_{1}} \times \cdots \times H_{\omega_{s}};$$
a contradiction.

To prove $(ii)$ let $\mathcal{H}$ be the subgroup $\mathbb{Z}^{((H_{\omega_1} \setminus G) \times \cdots \times (H_{\omega_s} \setminus G))} \rtimes H^s.$ The index $[\mathcal{G}:\mathcal{H}]$ is $m^s$. 

Let $\sigma = (1 \, 2 \, \cdots s) \in Sym(\{1, 2, \dots, s\})$. For each $i \in \{1, 2, \dots, s\}$ define the homomorphism $\rho_{(1)\sigma^i}: \mathcal{H} \rightarrow \mathcal{G}$ that extend the map

$$\rho_{(1)\sigma^i}: b^{(H_{\omega_1} g_{1}, \cdots ,H_{\omega_s} g_{s})} \mapsto
\begin{cases}
e,  \textrm{ if } (g_1,\dots,g_s)\notin H^s \\
b^{(H_{\omega_1} g_{1}^{f_{(1)\sigma^i}}, \cdots ,H_{\omega_s} g_{s}^{f_{(s)\sigma^i}})}, \textrm{ otherwise.}
    \end{cases}$$
$$(h_1,\dots,h_s) \mapsto (h_1^{f_{(1)\sigma^i}},\dots, h_s^{f_{(s)\sigma^i}}), \, (h_1,\dots,h_s) \in H^s$$
and the homomorphism $\mu: \mathcal{G} \rightarrow \mathcal{G}$ that extend the map
\begin{align*}
    \mu: b^{(H_\omega,\dots,H_\omega)}&\mapsto (x,\dots,x), \textrm{ where } o(x)=\infty \textrm{ and } x\in G\\
    (g_1,\dots,g_s)&\mapsto e, (g_1,\dots,g_s) \in G^s.
\end{align*}

We claim that the state-closed representation of $\mathcal{G}$ defined by the data $((m^s, \dots, m^s, 1), (\mathcal{H}, \dots, \mathcal{H}, \mathcal{G}), \{\rho_{1}, \dots, \rho_{s}, \mu\})$ is faithful.

Consider $K \leq \mathcal{H}$ with $K \vartriangleleft \mathcal{G}$, $K^{\rho_{i}} \leq K$, $i = 1, 2, \dots, s$, and $K^{\mu} \leq K$. Firstly, since $G^s$ is state-closed with respect to the tuple 
$$((m^s, \dots, m^s), (H, \dots, H), \{\rho_{1}|_{H}, \dots, \rho_{s}|_{H} \})$$ 
then $K \leq \mathbb{Z}^{((H_{\omega_1} \setminus G) \times \cdots \times (H_{\omega_s} \setminus G))}$. Let
$$y=b^{n_1(H_{\omega_1}h_{11},\dots,H_{\omega_s}h_{1s}) + n_2(H_{\omega_1}h_{21},\dots,H_{\omega_s}h_{2s}) +\dots + n_r(H_{\omega_1} h_{r1},\dots,H_{\omega_s} h_{rs})}$$
be an element of $K \setminus \{1\}$. By definition of $\mu$, we have $n_{1}+n_{2}+ \cdots + n_{r} = 0$. We can consider $r > 1$ minimal. Due to the normality of $K$, we can assume that $h_{11} = \cdots = h_{1s} = 1$. Since $\lambda$ is injective and $r$ is minimal, it follows that $y^{\rho_{1}^t} \neq 1$ and 
$$|y^{\rho_{1}^{n}}| = |b^{n_1(H_\omega,\dots,H_\omega)+ n_2(H_\omega k_{2, 1},\dots,H_\omega k_{2, s}) + \dots + n_r(H_\omega k_{r, 1},\dots,H_\omega k_{r,s})}| = r$$
for any $n = 0, 1, 2, \dots$. Then 
$$(h_{2 1}, \dots, h_{2s}), \dots, (h_{r 1}, \dots, h_{r s}) \in H_{\omega_{1}} \times \cdots \times H_{\omega_{s}};$$
a contradiction.

Now, let us prove that $\mathcal{G}$ is finite-state.  First, consider the group 
$$\mathcal{G}=B^{((H_{\omega_1} \setminus G) \times \cdots \times (H_{\omega_s} \setminus G))}\rtimes G^s.$$
Since $G$ is finite-state, $G^s$ is also finite-state. Let $G^s$ be finite-state with the choice of transversals $T_1,\dots, T_s$ of $H_1,\dots, H_s$ in $G$, respectively, where $t_{0i}=e$ is a representative of $H_i$,  for $i=1,\dots, s$, and define
$$\dot{B}=\{b^{(H_{\omega_1},\dots,H_{\omega_s})}\mid b\in {B}\}\simeq_{\phi} {B}.$$
For each $b\in B$, let $\dot{b}$ be the element of $\dot{B}$ satisfying $(\dot{b})^\phi=b$. Let $\dot{B}\times \prod_{i=1}^s T_{ij}$ be a transversal of $\mathcal{H}_{j} = \mathcal{H}$ in $\mathcal{G}$, where $j = 1, \dots, s$. Since $\mathcal{G}$ is generated by $\dot{B}\cup G^s$, it suffices to show that these elements are finite-state.

Consider the representation $\varphi$ of $G^s$ given by $g^\varphi=(t_j \mapsto g_{t_j})_{j=1}^s \cdot \pi$, where $t_j \in \prod_{i=1}^s T_{ij}$ and $\pi \in Sym(\dot{\bigcup}_{j=1}^{s} \prod_{i=1}^s T_{ij})$. If $\dot{x}\in \dot{B}$, we have
$$\dot{x}^{\dot{\varphi}}= \left( (\dot{b},t_{j})\mapsto \left\{\begin{array}{lr}
        e, & \text{ if } t_{j}=(e,\dots,e)\\
        \dot{x}, & \text{ otherwise}
        \end{array}\right\}\right)_{\dot{b} \in \dot{B}, \, j=1, \dots, s}((\dot{b},t_{j})\mapsto (\dot{xb}, t_{j})),$$
and if $g\in G^s$, then
$$g^{\dot{\varphi}}=((\dot{b},t)\mapsto \dot{b}g_{t_{j}}\dot{b}^{-1})_{\dot{b} \in \dot{B}, \, j=1, ..., s}((\dot{b},t_{j})\mapsto (\dot{b}, t_{j}^{\pi})).$$
In the first case, the set of  states of $\dot{b}^{\dot{\varphi}}$ is $\{e,\dot{b}\}$. In the second case, the set of states of $g^{\dot{\varphi}}$ is $\{ \dot{b}g_{t_j}\dot{b}^{-1} \mid t_j \in \prod_{i=1}^s T_{ij} \}$, which is contained in a finite set $U^{B}$, where $(U^{B})^{\varphi} \subset U^{B}$. Then $\mathcal{G}$ is finite-state.

The proof that 
$$\mathcal{G}={\mathbb{Z}}^{((H_{\omega_1} \setminus G) \times \cdots \times (H_{\omega_s} \setminus G))}\rtimes G^s$$
is finite-state is similar.  Now, it is necessary to replace $\dot{B}$ by $\dot{Z}$ defined by
$$\dot{Z}=\{a^{(H_{\omega_1},\dots,H_{\omega_s})}\mid a\in {\mathbb{Z}}\}\simeq_{\phi} {\mathbb{Z}}.$$
Here, we denote $\dot{a}$ as the element of $\dot{Z}$ such that $(\dot{a})^{\phi} = a$. Let $\prod_{i=1}^s T_{ij}$ be a transversal of $\mathcal{H}_{j} = \mathcal{H}$ in $\mathcal{G}$, where $j = 1, \dots , s$, and let $\{(e, \dots, e)\}$ be a transversal of $\mathcal{G}$ in $\mathcal{G}$. Since $\mathcal{G}$ is generated by $\dot{Z}\cup G^s$, it suffices to show that these elements are finite-state.

Consider the representation $\varphi$ of $G^s$ given by $g^\varphi=(t_{j}\mapsto g_{t_{j}})_{j = 1}^{s}\cdot \pi$, where $t_j \in \prod_{i=1}^s T_{ij}$ and $\pi\in Sym(\dot{\bigcup}_{j=1}^{s} \prod_{i=1}^s T_{ij})$. If $g\in G^s$, then
$$g^{\dot{\varphi}} = ((t_j \mapsto g_{t_j})_{j=1}^s, (e, \dots, e) \mapsto (e, \dots, e)) \cdot \pi (s\cdot m^s + 1).$$
If $\dot{a} \in \dot{Z}$, we have
$$\dot{a}^{\dot{\varphi}}= \left(\left(t_j \mapsto \left\{\begin{array}{lr}
        \dot{a}, & \text{ if } t_j =(e,\dots,e)\\
        (e, \dots, e), & \text{ otherwise}
        \end{array}\right\} \right)_{j=1}^s, (e, \dots, e) \mapsto {\bf x} \right),$$
where $a^{\mu} = (x, \dots, x) = {\bf x}$. In the first case, the set of states of $g^{\dot{\varphi}}$ is $\{(e, \dots, e), g_{t_j} \mid t_j \in \prod_{i=1}^s T_{ij} \}$, which is contained in a finite set $U$, where $U^{\varphi} \subset U$. In the second case, the set of  states of $\dot{a}^{\dot{\varphi}}$ is $\{e,\dot{a}\} \cup Q(\dot{x})$, which is finite. Then $\mathcal{G}$ is finite-state. 
\end{proof}

\begin{corollary} \label{cor4.2}
Let $G$ be a non-torsion self-similar group of orbit-type $(m_{1}, \dots, m_{s})$, where $m = [G : \cap_{i=1}^{s}H_{i}]$. Then, given any integer $l\geq 1$, the group
 $$(\mathbb{Z}^l)^{((H_{\omega_1} \setminus G) \times \cdots \times (H_{\omega_s} \setminus G))}\rtimes G^s$$ 
is self-similar of orbit-type $(\underbrace{m^s, \dots, m^s}_{sl \, \text{times}}, \underbrace{1,\dots,1}_{l \,\, \text{times}})$.

\end{corollary}

\begin{proof}

 By applying Theorem \ref{T3}, which establishes that $$(\mathbb{Z})^{((H_{\omega_1} \setminus G) \times \cdots \times (H_{\omega_s} \setminus G))}\rtimes G^s$$ is a self-similar group of orbit-type $(m^s, \dots , m^s, 1)$, and utilizing the process of concatenation as described in Proposition \ref{P3.2}, we deduce that $$(\mathbb{Z}^2)^{((H_{\omega_1} \setminus G) \times \cdots \times (H_{\omega_s} \setminus G))}\rtimes G^s$$ is a self-similar group of orbit-type $(m^s, \dots, m^s, 1, 1)$. 
\noindent We apply the same argument recursively $l-1$ times to derive the desired result.


\end{proof}

Theorem B follows from Theorem \ref{T3}, Corollary \ref{cor4.2} and the process of concatenation (Proposition \ref{P3.2}). 


\begin{corollary}\label{C1}
Let $A$ be a finitely generated abelian group. Then the group $A \wr (C_{2} \wr \mathbb{Z})$ is finite-state. In particular, the groups $C_{2} \wr (C_{2} \wr \mathbb{Z})$ and $\mathbb{Z} \wr (C_{2} \wr \mathbb{Z})$ are finite-state.
\end{corollary}

\begin{proof}
Let $G$ and $A$ be groups, $K \leq G$, and  $H_1, \dots, H_n$ be subgroups of $G$ such that $K \cap \cap_{i=1}^{n} H_{i} = 1$. Consider the group $$\mathcal{G}=A^{((H_1\backslash G)\times\dots\times (H_n\backslash G))}\rtimes G^n.$$
and the homomorphism
\begin{align*}
      \varphi: A\wr K &\rightarrow \mathcal{G}& \\
      \prod_{i=1}^s a_i^{n_{i1} h_{i1} + \cdots + n_{ir_{i}} h_{ir_{i}}}h & \mapsto \prod_{i=1}^s a_i^{n_{i1}(H_1 h_{i1}, \cdots, H_n h_{i1}) + \cdots + n_{ir_{i}}(H_1 h_{ir_{i}}, \cdots, H_n h_{ir_{i}})}(h, \dots, h)&
      \end{align*}
Since $K \cap \cap_{i=1}^{n} H_{i} = 1$, the homomorphism $\varphi$ is injective.

Now, let $f: H \rightarrow G$ be a virtual endomorphism, and $g$ be an element of $G$. Then
    \begin{align*}
      f_g=g^{-1}fg: H_G &\rightarrow G& \\
      h&\mapsto h^{g^{-1}fg}&
      \end{align*}
is a virtual endomorphism, where $H_G$ is the normal core of $H$ in $G$. Take $L \lhd G$ with $L$ $f_g$-invariant. If $x\in L$, then
   $$x^f=x^{gg^{-1}fgg^{-1}}=((x^{g})^{f_g})^{g^{-1}}\in L$$
and $L$ is also $f$-invariant.

Let $H_{\omega}$ be the parabolic subgroup of $f|_{H_{G}}$,
and $H_{\omega_{g}}$ be the parabolic subgroup of $f_{g}$. Then $$H_\omega^{gf_g}=H_\omega^{gg^{-1}fg}=H_\omega^{fg}=H_\omega^{g},$$ 
which implies $H_\omega^{g} \leq H_{\omega_g}$. Now, if $x \in H_{\omega_g} \leq H_{G}$, there exists $y \in H_{G}$ such that $x=y^{g}$, and
  $$x^{f_{g}^n}=x^{g^{-1}f^ng}=y^{f^ng},$$
so $y \in H_\omega$ and $x\in H_\omega^{g}$. Thus $H_{\omega}^g = H_{\omega_{g}}$. \\

We are ready to prove the corollary. Consider $G= C_2\wr \mathbb{Z} = \langle a \rangle \wr \langle x \rangle $ and the simple endomorphism 
\begin{align*}
      f:H=G'\langle x \rangle &\rightarrow G& \\
      [a,x]&\mapsto a&\\
      x&\mapsto x&
      \end{align*}
Since $H \lhd G$, then $H_G = H$ and 
\begin{align*} f_a=afa: H &\rightarrow G& \\
      [a,x]&\mapsto [a,x]^{afa}=a&\\
      x&\mapsto x^{afa}=([a,x]x)^{fa}=(ax)^{a}=a^xx& \end{align*}
is a simple endomorphism such that $H_{\omega_a} = H_{\omega}^a=\langle x \rangle ^a = \langle [a,x]x \rangle$. Note that $([a,x]x)^{f_a}=aa^xx=[a,x]x$, then $H_{\omega}\cap H_{\omega_a} = 1$. By the first paragraph of this proof, $A \wr (C_2 \wr \mathbb{Z})$ is a subgroup of 
$$\mathcal{G} = A^{(H_{\omega} \setminus (C_{2} \wr \mathbb{Z})) \times (H_{\omega_a} \setminus (C_{2} \wr \mathbb{Z}))} \rtimes (C_{2} \wr \mathbb{Z})^2.$$
By Corollary \ref{cor4.2}, $\mathcal{G}$ is state-closed and finite-state, so $A \wr (C_2\wr\mathbb{Z})$ is also finite-state.
\end{proof}

In case $G=C_2\wr\mathbb{Z}$, by Theorem \ref{T3} we have that $C_2^{(H_\omega\backslash G)\times (H_{\omega_a}\backslash G)}\rtimes G^2\simeq\langle \gamma, \beta_1,\beta_2,\alpha_1,\alpha_2\rangle\leq\mathcal{A}_{16}$ where

\begin{align*}
    \gamma=&(1\,5)(2\,6)(3\,7)(4\,8)(9\,13)(10\,14)(11\,15)(12\,16)\\
    \beta_1=&(1\,2)(3\,4)(5\,6)(7\,8)(9\,10)(11\,12)(13\,14)(15\,16)\\
    \beta_2=&(1\,3)(2\,4)(5\,7)(6\,8)(9\,11)(10\,12)(13\,15)(14\,16)\\
    \alpha_1=&(\alpha_1,\beta_1\alpha_1,\alpha_1,\beta_1\alpha_1,[\gamma,\beta_1\alpha_1]\beta_1\alpha_1,\alpha_1,[\gamma,\beta_1\alpha_1]\beta_1\alpha_1,\alpha_1\beta_1,\alpha_1,\alpha_1\beta_1,\alpha_1,\alpha_1\beta_1,\alpha_1^{\beta_1},\\
    &\alpha_1\beta_1,\alpha_1^{\beta_1})\\
    \alpha_2=&(\alpha_2,\alpha_2,\alpha_2\beta_2,\alpha_2\beta_2,\alpha_2[\gamma,\alpha_2],\alpha_2[\gamma,\alpha_2],e,e,\alpha_2\beta_2,\alpha_2\beta_2,\alpha_2^{\beta_2},\alpha_2^{\beta_2},\alpha_2\beta_2[\gamma,\alpha_2\beta_2],\\
    &\alpha_2\beta_2[\gamma,\alpha_2\beta_2],e,e)
\end{align*}

and $\mathbb{Z}^{(H_\omega\backslash G)\times (H_{\omega_a}\backslash G)}\rtimes G^2\simeq\langle y,a_1,a_2,x_1,x_2 \rangle\leq\mathcal{A}_8$ where
\begin{align*}
    y=&(y,e,e,e,y,e,e,e,x_1)\\
    a_1=&(1\,2)(3\,4)(5\,6)(7\,8)\\
    a_2=&(1\,3)(2\,4)(5\,7)(6\,8)\\
    x_1=&(x_1,a_1x_1,x_1,a_1x_1,x_1a_1,a_1x_1a_1,x_1a_1,a_1x_1a_1,e)\\
    x_2=&(x_2,x_2,a_2x_2,a_2x_2,x_2a_2,x_2a_2,a_2x_2a_2,a_2x_2a_2,e).
\end{align*}

\section{Finite-state wreath product of type $\mathbb{Z} \wr G$}


In this section, we prove Theorem C for the particular case $A = \mathbb{Z}$, as the general case follows in a similar manner. We assume $G$ is a self-similar finite-state non-torsion group with respect to the data $$(\textbf{m}, \textbf{H}, \textbf{F}) = ((m, \dots, m), (H, \dots, H), (f_1,\dots,f_s)),$$ where $H$ is the core of $\cap_{i=1}^{s}H_{i}$ in $G$ and $m = [G:H]$. Consider $\mathcal{G}$ the semidirect product
$$\mathcal{G}=\mathbb{Z}^{(H_\omega \backslash G\times\dots\times H_\omega \backslash G )}\rtimes G^s,$$
and its subgroup
$$\mathcal{H}=\mathbb{Z}^{(H_\omega \backslash G\times\dots\times H_\omega \backslash G )}\rtimes (H\times\dots\times H)$$ 
with index $[G:H]^s$. Set the homomorphisms $\rho: \mathcal{H} \rightarrow \mathcal{G}$, $\tau:\mathcal{G} \rightarrow \mathcal{G}$, and $\mu:\mathcal{G} \rightarrow \mathcal{G}$ that extend the maps
\begin{align*}
    \rho: a^{(H_\omega g_1,\dots, H_\omega g_s)}&\mapsto\begin{cases}
       a^{(H_\omega g_1^{f_1},\dots, H_\omega g_s^{f_s})}, \textrm{ if } (g_1,\dots,g_s)\in H\times\dots\times H\\
       e,  \textrm{ if } (g_1,\dots,g_s)\notin H\times\dots\times H.
    \end{cases}\\
    (h_1,\dots,h_s)&\mapsto (h_1^{f_1},\dots, h_s^{f_s}), \, (h_1,\dots,h_s) \in H\times\dots\times H;
\end{align*}

\begin{align*}
    \tau: a^{(H_\omega g_1,H_\omega g_2\dots, H_\omega g_s)}&\mapsto
       a^{(H_\omega g_2,H_\omega g_3\dots, H_\omega g_{s}, H_\omega g_{1})}, (g_1,\dots,g_s) \in G^s,\\
    (g_1,\dots,g_s)&\mapsto (g_2, g_3\dots, g_{s}, g_{1}), (g_1,\dots,g_s) \in G^s;
\end{align*}
and
\begin{align*}
    \mu: a^{(H_\omega,\dots,H_\omega)}&\mapsto (x,\dots,x) \textrm{ where } o(x)=\infty \textrm{ and } x\in G\\
    (g_1,\dots,g_s)&\mapsto e, (g_1,\dots,g_s) \in G^s.
\end{align*}

    Now, consider the representation $\varphi:\mathcal{G}\rightarrow \mathcal{A}_{n}$ induced by the data $(\overline{\textbf{m}}, \overline{\textbf{H}}, \overline{\textbf{F}}) = ((m^s, 1, 1), (\mathcal{H}, \mathcal{G}, \mathcal{G}), (\rho, \tau, \mu))$, where $n = m^s +2$ and $m = [G : H]$. So, $\varphi$ is a state-closed representation of $\mathcal{G}$ with  kernel $\overline{\textbf{F}}$-core($\overline{\textbf{H}}$).

Note that $\mathbb{Z}^{H_{w} \setminus G} \rtimes G$ is isomorphic to a subgroup of $\mathcal{G}$; in fact, it is enough consider the isomorphism
\begin{align*}
    \delta: \mathbb{Z}^{(H_{\omega} \setminus G)} \rtimes G&\rightarrow D = \left\langle a^{(H_{\omega},\dots, H_{\omega})}, (g,\dots, g) \mid g \in G \right\rangle \leq \mathcal{G}\\
    a^{H_{\omega}g_{1} + \dots + H_{\omega}g_{n}}g&\mapsto a^{(H_{\omega}g_{1}, \dots, H_{\omega}g_{1}) + \dots + (H_{\omega}g_{n}, \dots, H_{\omega}g_{n})}(g,\dots,g).
\end{align*}

\begin{thm}\label{T2.1}
    The group $\mathbb{Z}^{(H_\omega\backslash G)}\rtimes G$ is isomorphic to a subgroup of $\mathcal{G}^{\varphi}$. Moreover, $\mathcal{G}^{\varphi}$ is finite-state.
\end{thm}

\begin{proof} Let $\varphi:\mathcal{G}\rightarrow \mathcal{A}_n$ be the self-similar representation of $\mathcal{G}$ induced by the data $(\overline{\textbf{m}}, \overline{\textbf{H}}, \overline{\textbf{F}}) = ((n, 1, 1), (\mathcal{H}, \mathcal{G}, \mathcal{G}), (\rho, \tau, \mu))$ with $\overline{\textbf{F}}$-core($\overline{\textbf{H}}) = K$. Firstly we prove that $\mathbb{Z}^{(H_\omega\backslash G)}\rtimes G$ is a subgroup of $\mathcal{G}^{\varphi}$. For this it is enough show that $D \cap K = 1$.

If $y\in D\cap K$ then
    $$y=a^{n_1(H_\omega h_1,\dots,H_\omega h_1 )+\dots + n_r(H_\omega h_r,\dots,H_\omega h_r )}(h,\dots,h).$$
Since $G$ is a self-similar group and $K$ is normal in $G$, we can assume that $h_1=e$ and by applications of $\rho$, $\tau$ we can take $(h,\dots,h)=(e,\dots,e)$. Thus
$$y=a^{n_1(H_\omega,\dots,H_\omega)+\dots + n_r(H_\omega h_r,\dots,H_\omega h_r )}.$$

Suppose that $r > 1$ is minimal. By Proposition \ref{3.3}
$$y^{\tau^{\epsilon_{1}}\rho \tau^{\epsilon_{2}}\rho... \tau^{\epsilon_{n}}\rho}$$
is non trivial for any positive integer $n$ and any $0 \leq \epsilon_{1}, ..., \epsilon_{n} \leq s - 1$. Reordering the indexes, if necessary, let $n$ be the positive integer such that
$$y^{\tau^{\epsilon_{1}}\rho \tau^{\epsilon_{2}}\rho... \tau^{\epsilon_{n}}\rho} = a^{n_1(H_\omega,\dots,H_\omega)+ n_2(H_\omega h_{2, 1},\dots,H_\omega h_{2, s}) + \dots + n_l(H_\omega h_{l, 1},\dots,H_\omega h_{l,s})}$$
and $l > 1$ is minimal with $(h_{2, 1}, ..., h_{2, s}), ..., (h_{l, 1}, ..., h_{l, s}) \notin H_{\omega}^s$. Then
$$|(y^{\tau^{\epsilon_{1}}\rho \tau^{\epsilon_{2}}\rho... \tau^{\epsilon_{n}}\rho})^{\tau^{\epsilon_{n+1}}\rho \tau^{\epsilon_{n+2}}\rho... \tau^{\epsilon_{n+k}}\rho}| = l,$$
for any positive integer $k$ and any $0 \leq \epsilon_{n+1},\dots, \epsilon_{n+k} \leq s - 1$. Thus,
$$h_{2, 1}, \dots, h_{2, s}, \dots, h_{l, 1}, \dots, h_{l, s} \in H_{\omega} = \langle K \leq \cap_{i=1}^{s} H_{i} \mid K^{f_{i}} \leq K, \text{ for all }i \rangle$$
and this final contradiction prove the first assertion.
\,

Now we prove that the group $\mathcal{G}^\varphi$ is finite-state whenever $G$ is.

\,

    Let 
    $$\mathcal{H}=\mathbb{Z}^{(H_\omega\backslash G\times\dots\times H_\omega\backslash G)}\rtimes {H}.$$
As $G$ is finite-state, $G^s$ is also finite-state. Now, let $G^s$ be finite-state for the choice of transversals $T_1,\dots, T_s$ of $H_1,\dots, H_s$ in $G$, respectively, with $t_{0i}$ representants of $H_i$, $i=1,\dots, s$, and define
$$\dot{Z}=\{a^{(H_\omega,\dots,H_\omega)}\mid a\in\mathbb{Z}\}\simeq_{\phi} \mathbb{Z}.$$
Write $\dot{a}$ the element of $\dot{Z}$ such that $(\dot{a})^{\phi} = a$. Let $\prod_{i=1}^s T_{i}$ be a transversal of $\mathcal{H}$ in $\mathcal{G}$ and let $1$ be a transversal of $\mathcal{G}$ in $\mathcal{G}$. Since $\mathcal{G}$ is generated by $\dot{Z}\cup G^s$, it suffices to show that these elements are finite-state.

Consider the representation $\varphi$ of $G^s$ given by $g^\varphi=(t\mapsto g_{t}) \cdot \pi$, where $t \in \prod_{i=1}^s T_{i}$ and $\pi\in Sym(\prod_{i=1}^s T_{i})$. If $g\in G^s$, then
$$g^{\dot{\varphi}} = ((t \mapsto g_{t}), 1 \mapsto (g^{\tau})^{\varphi}, 1 \mapsto 1) \cdot \pi (m^s + 1)(m^s + 2).$$
If $\dot{a} \in \dot{Z}$, we have
$$\dot{a}^{\dot{\varphi}}= \left(t \mapsto \left\{\begin{array}{lr}
        \dot{a}^{\varphi}, & \text{ if } t =(e,\dots,e)\\
        1, & \text{ otherwise}
        \end{array}\right\}, 1 \mapsto \dot{a}^{\varphi}, 1 \mapsto {\bf x}^{\varphi} \right),$$
where $a^{\mu} = (x, \dots, x) = {\bf x}$. In the first case, the set of states of $g^{\dot{\varphi}}$ is $\{1, g_{t} \mid t \in \prod_{i=1}^s T_{i} \} \cup Q((g^{\tau})^{\varphi}) \cup \cdots \cup Q((g^{{\tau}^{s-1}})^{{\varphi}})$, which is contained in a finite set $U$, where $U^{\varphi} \subset U$ (Note that $g^{\tau^s} = g$). In the second case, the set of  states of $\dot{a}^{\dot{\varphi}}$ is $\{e,\dot{a}^{\varphi}\} \cup Q({\bf x}^{\varphi})$, which is finite. Then $\mathcal{G}$ is finite-state.






\end{proof}








Define the $k$-iterated wreath product of copies of $\mathbb{Z}$ inductively as follows: $W_{1}(\mathbb{Z}) = \mathbb{Z}$, and $W_{k}(\mathbb{Z}) = \mathbb{Z} \wr W_{k-1}(\mathbb{Z})$ for $k > 1$.

\begin{corollary} \label{cor5.2}
 Let $G$ be a finite-state self-similar group. Then the group $\mathbb{Z}\wr G$ has a finite-state representation. In particular, $W_{k}(\mathbb{Z})$ is finite-state for any $k \geq 1$.
\end{corollary}

\begin{proof}
   By Theorem A, we can take a self-similar representation of $G$ whose parabolic subgroup $H_\omega$ is trivial. Then $\mathcal{G}=\mathbb{Z}^{(H_\omega \backslash G\times\dots\times H_\omega \backslash G )}\rtimes G^s\simeq \mathbb{Z}\wr G^s$, and by Theorem \ref{T2.1}, we have $D\simeq \mathbb{Z}\wr G\simeq D^\varphi$, which is also finite-state.
\end{proof}

\section{Finite-state representation of degree two of the group $\mathbb{Z}\wr (\mathbb{Z} \wr \mathbb{Z})$}

Since the group $\mathbb{Z} \wr \mathbb{Z}$ is finite-state and self-similar, then by Corollary \ref{cor5.2}, the group $\mathbb{Z} \wr (\mathbb{Z} \wr \mathbb{Z})$ is finite-state.

\begin{theorem}
The group $\mathbb{Z} \wr (\mathbb{Z} \wr \mathbb{Z})$ is finite-state of degree $2$.     
\end{theorem}

\begin{proof}
By Example \ref{ex3.2}, $\mathbb{Z} \wr \mathbb{Z}$ is self-similar of orbital-type $(2, 2, 2)$. Applying Theorem \ref{T2.1}, the group $\mathbb{Z} \wr (\mathbb{Z} \wr \mathbb{Z})$ is isomorphic to a subgroup of the group $(\mathbb{Z} \wr (\mathbb{Z} \wr \mathbb{Z})^3)^{\varphi}$ which has orbital-type $(8, 1, 1)$, where 
$$\varphi: \mathbb{Z} \wr (\mathbb{Z} \wr \mathbb{Z})^3 = \langle a \rangle \wr \left(\prod_{i = 1}^{3} (\langle y_{i} \rangle \wr \langle x_{i} \rangle)\right) \rightarrow \mathcal{A}_{10}$$ is a homomorphism induced by the homomorphisms $$\rho: \mathcal{H} = \langle a \rangle^{\langle y_{i}, x_{i} \mid i = 1, 2, 3 \rangle} \langle y_{1}, y_{2}, y_{3} \rangle^{\langle x_{1}, x_{2}, x_{3} \rangle} \langle x_{1}^2, x_{2}^2, x_{3}^2 \rangle \rightarrow \mathcal{G},$$
$$\tau:\mathcal{G} \rightarrow \mathcal{G}, \mu:\mathcal{G} \rightarrow \mathcal{G},$$
that extend the maps
\noindent
\begin{align*}
    \rho: a^{y_{1}^{x_{1}^{2r}}}&\mapsto a^{y_{i}^{x_{1}^{r}}} \\
    a^{y_{2}^{x_{2}^{2r}}}&\mapsto a \\
    a^{y_{1}^{x_{1}^{2r+ 1}}}&\mapsto a\\
    a^{y_{2}^{x_{2}^{2r + 1}}}&\mapsto a^{y_{2}^{x_{2}^{r}}}, \\
    a^{y_{3}}&\mapsto a^{x_{3}} \\
    a^{x_{1}^{2r}}&\mapsto a^{x_{1}^{r}} \\
    a^{x_{1}^{2r + 1}}&\mapsto e \\
    a^{x_{2}^{2r}}&\mapsto a \\
    a^{x_{2}^{2r+1}}&\mapsto e \\
    a^{x_{3}^{2r}}&\mapsto a \\
    a^{x_{3}^{2r+1}}&\mapsto e \\
    y_{1}^{x_{1}^{2r}}&\mapsto y_{1}^{x_{1}^{2r}}\\
    y_{1}^{x_{1}^{2r+1}}&\mapsto e\\
    y_{2}^{x_{2}^{2r+1}}&\mapsto y_{2}^{x_{2}^{2r}}\\
    y_{2}^{x_{2}^{2r}}&\mapsto e\\
    x_{i}^{2r}&\mapsto x_{i}^{r}, i = 1, 2\\
    x_{3}^{2r} &\mapsto e;
\end{align*}
\begin{multicols}{2}
\noindent \begin{align*}
    \tau: a^{y_{i}}&\mapsto
       a^{y_{(i)\sigma}}, \text{ where } \sigma = (1 \, 2 \, 3)\\
       a^{x_{i}}&\mapsto
       a^{x_{(i)\sigma}}\\
    y_{i}&\mapsto y_{(i)\sigma}\\
    x_{i}&\mapsto x_{(i)\sigma}\\
\end{align*}
\noindent \begin{align*}
    \mu: a&\mapsto y_{1}\\
    y_{i}, x_{j} &\mapsto e, i,j = 1, 2, 3.
\end{align*}
\end{multicols}
Consider the transversals $T_1 = \{e, x_2, x_3, x_2x_3, x_1x_2x_3, x_1x_3, x_1x_2, x_1\}$
 of $\mathcal{H}$ in $\mathcal{G}$ and $T_2=T_3=\{e\}$ of $\mathcal{G}$ in $\mathcal{G}$. Thus the group $\mathcal{G}^{\varphi}$ is the group generated by the elements
$$a=(a,e,e,e,e,e,e,e,a,y_1)$$ $$y_1=(y_1,x_1,y_1,y_1,e,e,e,e,y_3,e),$$ $$y_2=(x_2,x_2,x_2,x_2,x_2,x_2,x_2,x_2,y_1,y_1),$$ $$y_3=(e,e,y_3,y_3,y_3,y_3,e,e,y_2,y_1),$$
$$x_1=(e,e,e,e,x_1,x_1,x_1,x_1,x_3,e)(1\,8)(2\,7)(3\,6)(4\,5),$$
$$x_2=(e,x_2,e,x_2,x_2,e,x_2,e,x_3,e)(1\,2)(3\,4)(5\,6)(7\,8),$$
$$x_3=(e,e,e,e,e,e,e,e,x_1,e)(1\,3)(2\,4)(5\,7)(6\,8).$$
By a degree deflation we obtain that the group generated by the elements
$${\bf a}=(\textcolor{green}{(}\textcolor{blue}{(}\textcolor{red}{(}{\bf a}, e\textcolor{red}{)},\textcolor{red}{(}e,e\textcolor{red}{)}\textcolor{blue}{)},\textcolor{blue}{(}\textcolor{red}{(}e,e\textcolor{red}{)},\textcolor{red}{(}e,e\textcolor{red}{)}\textcolor{blue}{)}\textcolor{green}{)},\textcolor{green}{(}{\bf a},{\bf y_1}\textcolor{green}{)}),$$
$${\bf y_1}=(\textcolor{green}{(}\textcolor{blue}{(}\textcolor{red}{(}{\bf y_1},{\bf x_1}\textcolor{red}{)},\textcolor{red}{(}{\bf y_1},{\bf y_1}\textcolor{red}{)}\textcolor{blue}{)},\textcolor{blue}{(}\textcolor{red}{(}e,e\textcolor{red}{)},\textcolor{red}{(}e,e\textcolor{red}{)}\textcolor{blue}{)}\textcolor{green}{)},\textcolor{green}{(}{\bf y_3},e\textcolor{green}{)}),$$
$${\bf y_2}=(\textcolor{green}{(}\textcolor{blue}{(}\textcolor{red}{(}{\bf x_2},{\bf x_2}\textcolor{red}{)},\textcolor{red}{(}{\bf x_2},{\bf x_2}\textcolor{red}{)}\textcolor{blue}{)},\textcolor{blue}{(}\textcolor{red}{(}{\bf x_2},{\bf x_2}\textcolor{red}{)},\textcolor{red}{(}{\bf x_2},{\bf x_2}\textcolor{red}{)}\textcolor{blue}{)}\textcolor{green}{)},\textcolor{green}{(}{\bf y_1},{\bf y_1}\textcolor{green}{)}),$$
$${\bf y_3}=(\textcolor{green}{(}\textcolor{blue}{(}\textcolor{red}{(}e,e\textcolor{red}{)},\textcolor{red}{(}{\bf y_3},{\bf y_3}\textcolor{red}{)}\textcolor{blue}{)},\textcolor{blue}{(}\textcolor{red}{(}{\bf y_3},{\bf y_3}\textcolor{red}{)},\textcolor{red}{(}e,e\textcolor{red}{)}\textcolor{blue}{)}\textcolor{green}{)},\textcolor{green}{(}{\bf y_1},{\bf y_2}\textcolor{green}{)})$$
$${\bf x_1}=(\textcolor{green}{(}\textcolor{blue}{(}\textcolor{red}{(}e,e\textcolor{red}{)}(1 \, 2),\textcolor{red}{(}e,e\textcolor{red}{)}(1\,2)\textcolor{blue}{)}(1 \, 2),\textcolor{blue}{(}\textcolor{red}{(}{\bf x_1},{\bf x_1}\textcolor{red}{)},\textcolor{red}{(}{\bf x_1},{\bf x_1}\textcolor{red}{)}\textcolor{blue}{)}\textcolor{green}{)}(1 \, 2),\textcolor{green}{(}{\bf x_3},e\textcolor{green}{)})$$
$${\bf x_2}=(\textcolor{green}{(}\textcolor{blue}{(}\textcolor{red}{(}e,{\bf x_2}\textcolor{red}{)}(1 \, 2),\textcolor{red}{(}e,{\bf x_2}\textcolor{red}{)}(1 \, 2)\textcolor{blue}{)},\textcolor{blue}{(}\textcolor{red}{(}{\bf x_2},e\textcolor{red}{)}(1 \, 2),\textcolor{red}{(}{\bf x_2},e\textcolor{red}{)}(1 \, 2)\textcolor{blue}{)}\textcolor{green}{)},\textcolor{green}{(}{\bf x_3},e\textcolor{green}{)})$$
$${\bf x_3}=(\textcolor{green}{(}\textcolor{blue}{(}\textcolor{red}{(}e,e\textcolor{red}{)},\textcolor{red}{(}e,e\textcolor{red}{)}\textcolor{blue}{)}(1 \, 2),\textcolor{blue}{(}\textcolor{red}{(}e,e\textcolor{red}{)},\textcolor{red}{(}e,e\textcolor{red}{)}\textcolor{blue}{)}(1 \, 2)\textcolor{green}{)},\textcolor{green}{(}{\bf x_1},e\textcolor{green}{)})$$
is a finite-state subgroup of $\mathcal{A}_2$ and is isomorphic to $(\mathbb{Z} \wr (\mathbb{Z} \wr \mathbb{Z})^3)^{\varphi}$. Since 
$$\mathbb{Z} \wr (\mathbb{Z} \wr \mathbb{Z}) \simeq \langle a, y_1y_2y_3, x_1x_2x_3 \rangle \simeq \langle {\bf a}, {\bf y_1y_2y_3}, {\bf x_1x_2x_3} \rangle,$$
the result follows.
\end{proof}


\begin{thebibliography}{9}

\bibitem{BarSid} \textsc{L. Bartholdi and S. N. Sidki}, \textit{ Self-similar products of groups,} Groups, Geometry, and Dynamics, \textbf{14} (2020), 107--115.

\bibitem{BSZ} \textsc{L. Bartholdi, O. Siegenthaler, and P. Zalesskii}. \textit{The congruence subgroup problem for branch groups}. Israel Journal of Mathematics,  \textbf{187}-1, (2012), 419-
450.

\bibitem{BDS} \textsc{A. A. Berlatto, A. C. Dantas, and T. M. G. Santos}, \textit{Self-similarity of some soluble relatively free groups,} Arch. Math. \textbf{120} (2023), 361 – 371 .

\bibitem{BerSid} \textsc{A. Berlatto and S. N. Sidki}, \textit{Virtual endomorphisms of nilpotent groups}, Groups, Geometry, and Dynamics, \textbf{1} (2007), 21 -
46.


\bibitem{BS} \textsc{A. M. Brunner and S. N. Sidki}, \textit{Abelian state-closed subgroups of automorphisms of $m$-ary trees}\emph{,} Groups, Geometry, and Dynamics, 
\textbf{4} (2010), 455 - 471.

\bibitem{BruSid1} \textsc{ A. M. Brunner and S. N. Sidki}, \textit{ On the Automorphism Group of the One-Rooted Binary Tree,} Journal of Algebra, \textbf{195} (1997),  465--486.

\bibitem{BruSid2} \textsc{ A. M. Brunner and S. N. Sidki}, \textit{ The Generation of $GL(n,\mathbb{Z})$ by Finite State Automata,} International Journal of Algebra and Computation, \textbf{08} (1998),  127--139.

\bibitem{BruSid3} \textsc{ A. M. Brunner and S. N. Sidki}, \textit{ Wreath operations in the group of automorphisms of the binary tree,} Journal of Algebra, \textbf{257} (2002),  51--64.

\bibitem{D} \textsc{A. C. Dantas}, \textit{State-closed groups with bounded conjugacy classes}, Journal of Algebra and its Applications, \textbf{15}, 2016.

\bibitem{DSS} \textsc{A. C. Dantas, T M. G. Santos and  S. N. Sidki}, \textit{Intransitive self-similar groups}, Journal of Algebra, \textbf{567}, 2021, 564--581.

\bibitem{DS} \textsc{A. C. Dantas, S. N. Sidki}, \textit{On self-similarity of wreath products of abelian groups}, Groups Geom.
Dyn. \textbf{12}, (2018) 1061–1068.

\bibitem{PVV} \textsc{P. W. Gawron, V. V. Nekrashevych, and V.I. Sushchansky}, \textit{Conjugation in tree automorphism groups}, Int. J. Algebra Comput.,  \textbf{11}, (2001), no. 5, 529–-547.

\bibitem{G} \textsc{R. I. Grigorchuk}, \textit{On the Burnside problem for periodic groups}, Funct. Anal. Appl. \textbf{14} (1980) 41-43.

\bibitem{GuS} \textsc{N. Gupta, S. Sidki}, \textit{On the Burnside problem for periodic groups}, Math. Z. \textbf{182} (1983) 385–388.

\bibitem{K} \textsc{M. Kapovich}, \textit{Arithmetic aspects of self-similar groups}, Groups Geom. Dyn. \textbf{6}, (2012) 737–754.

\bibitem{KS} \textsc{D. Kochloukova, S.N. Sidki}, \textit{Self-similar groups of type $FP_n$}, Geom. Dedic. \textbf{204}, (2020) 241–264.

\bibitem{NK} \textsc{Mazurov, E.I. Khukhro}, \textit{The Kourovka Notebook - Unsolved Problems in Group Theory}, 18th ed., Institute of Mathematics, Russian Academy of Sciences, Siberrian Division, Novosibirsk, (2014), arXiv:1401.0300v3 [math.GR].

\bibitem{S} \textsc{S. Sidki}, \textit{Tree-wreathing applied to generation of groups by finite automata}, International Journal of Algebra and Computation \textbf{15} (2005) 1204–1212.

\bibitem{RS} \textsc{R. Skiper}, \textit{A constructive proof that the Hanoi towers group has non-trivial rigid kernel}, Topology Proceedings, \textbf{53}, (2018), 1 - 14.


\bibitem{NS} \textsc{ V. Nekrashevych and S. N. Sidki}, \textit{ Automorphisms of the binary
	tree: state-closed subgroups and dynamics of 1/2-endomorphisms}, Groups:
topological, combinatorial and aritmetic aspects, London Mathematical
Society Lecture Note Series, \textbf{311}. Cambridge University Press, (2004), 375--404.

\end{thebibliography}
\end{document}